\documentclass{article}
\usepackage{amssymb}
\usepackage{amsfonts}
\usepackage{amsmath}
\usepackage{amsthm}
\usepackage{mathrsfs}
\usepackage{epsfig}
\usepackage{graphicx}
\usepackage{multirow}
\usepackage[ruled,vlined]{algorithm2e}
\usepackage{bm}
\usepackage{hyperref}
\usepackage{spconf}
\usepackage{xcolor}
\usepackage{verbatim}

\allowdisplaybreaks

\newtheorem{theorem}{Theorem}[section]  
\newtheorem{definition}{Definition}

\newtheorem{proposition}{Proposition}

\let\OLDthebibliography\thebibliography
\renewcommand\thebibliography[1]{
  \OLDthebibliography{#1}
  \setlength{\parskip}{0pt}
  \setlength{\itemsep}{0pt plus 0.3ex}
}

\numberwithin{equation}{section}

\newcommand{\sgn}{{\rm sgn}}

\begin{document}\topmargin=0mm 
\sloppy

\title{Generalized Opinion Dynamics Model for Social Trust Networks in Opinion Maximization}
%
\name{Changxiang He, Jiayuan Zeng, Shuting Liu, Guang Zhang, Xiaofei Qin, Xuedian Zhang, Lele Liu }
\address{University of Shanghai for Science and Technology\\changxiang-he@163.com, 854789302@qq.com,\\2314914087@qq.com, g.zhang@usst.edu.cn, xiaofei.qin@usst.edu.cn\\obmmd\_zxd@163.com, leliu@usst.edu.cn}

\maketitle

\begin{abstract}
In this paper, we propose a generalized opinion dynamics model (GODM), which can dynamically compute each person's expressed opinion,
to solve the internal opinion maximization problem for social trust networks. In the model, we propose a new, reasonable and 
interpretable confidence index, which is determined by both person’s social status and the evaluation around him. By using the 
theory of diagonally dominant, we obtain the optimal analytic solution of the Nash equilibrium with maximum overall opinion. 
We design a novel algorithm to maximize the overall with given budget by modifying the internal opinions of people in the social 
trust network, and prove its optimality both from the algorithm itself and the traditional optimization algorithm-ADMM algorithms 
with $l_1$-regulations. A series of experiments are conducted, and the experimental results show that our method is superior to 
the state-of-the-art in four datasets. The average benefit has promoted $67.5\%$, $83.2\%$, $31.5\%$, and $33.7\%$ on four datasets, 
respectively.
\end{abstract}

\section{Introduction and motivation}

The study of how people form their opinions through social interactions with others has long been the subject of research 
in the field of social science. On social networks, users can express their opinions on a product or a real-time new event. 
By quantifying these opinions, and forming the expressed opinions through social interactions with others, the overall 
opinion of a network on the product or event can be determined. This allows one to quantitatively study interactions of 
opinions on a large scale, and to improve them through targeted interventions. Generally speaking, for the merchant, 
the overall opinion of the network on the product as high as possible is expected. In order to realize it, if possible,
the overall opinion can be adjusted through a certain budget. For the government, when a positive real-time new event 
occurs, it hopes that the overall opinion of this event as high as possible. Conversely, when a negative one happens,
the government hopes that the overall opinion as low as possible. It is obviously that the minimization of the overall 
opinion under the premise of a certain budget can be achieved by converting the original network to a negative network.
  
\begin{figure}[t]
    \centerline{\epsfig{figure=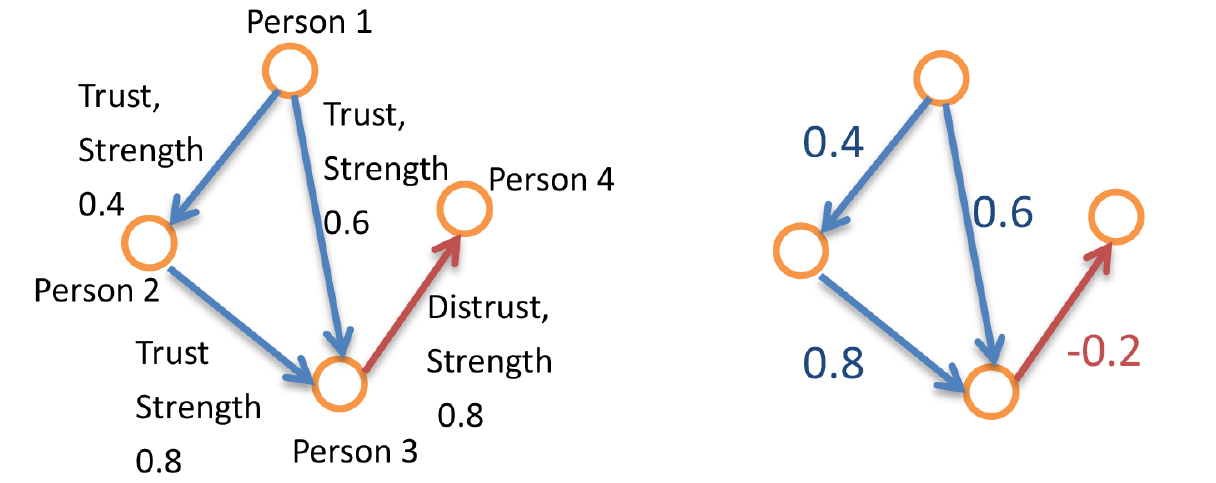,width=7cm}}
    \caption{A social trust network (left) and its graph representation (right).}
\label{fig:res2}
\end{figure}

\subsection{Background}
Social trust network is a relatively stable network of relationships among certain groups (people, enterprises
and organizations). A social trust network and its graph representation can be seen in Fig. \ref{fig:res2}.
Many opinion formation models have been proposed and studied based on the influence through social interaction, 
and the issue of opinion or influence maximization has been widely studied in \cite{2015Online} and \cite{2020Opinion}.
Chen et al. \cite{2015Online} proposed an algorithm to find $K$ nodes which perform the greatest common influence 
on the network. Kong et al. \cite{Liu2018} considered an issue called active opinion maximization, whose goal is 
to find a subset of seed users to maximize the overall opinion dissemination of the target product in a multi-round 
campaign. In addition to considering interventions that directly modify people's internal opinions, Abebe et al. 
\cite{2018Opinion} also considered interventions that to change people's sensitive to persuasion. The goal was to 
modify individual sensitivities to maximize the overall opinion at Nash equilibrium. And \cite{2020Opinion} proposed 
a new continuous-valued opinion dynamics model for social trust networks, which was more consistent with real scenarios 
and achieved good results. But there is still something can be improved in \cite{2020Opinion}, it only considered 
the weighted average of the opinions of person's internal opinions and those of the people around it. However, 
in real life, everyone's personalities is different. That is to say, the influence of other people's opinion is 
supposed to vary for different person. Therefore, we have made some improvements on the basis of \cite{2020Opinion} 
and achieved excellent results.


\subsection{Contributions in this paper}
\begin{itemize}\setlength{\itemsep}{1pt}
  \item We develop a generalized opinion dynamics model for social trust networks to precisely calculate the opinion 
everyone express. And we introduce a novel confidence index $\alpha_i$, which is determined by a person's social status 
and the evaluation of people around him. In the traditional dynamics model, a person's opinion is usually expressed as 
a weighted average of his internal opinion and the opinions of those around them. For different human groups, the range 
of $\alpha_i$ will be different. For example, when there are mostly young people in the group, the value of $\alpha_i$ 
tends to be relatively low, and when most people in the group are elderly, the value of $\alpha_i$ will be higher. From 
the theory and experiments, we can see that when $\alpha_i=1/2$, our model GODM is the same one as the model in \cite{2020Opinion}.
  \item By using the theory of diagonally dominant and Friedkin-Johnsen model, we give the analytic solution in GODM 
when the Nash equilibrium is reached.
  \item We introduce a novel algorithm and prove that the result is optimal. In addition, we introduce a mature optimization 
algorithm named ADMM, and solve the problem with special form derived from this algorithm-$l_1$ regulation and achieve 
the same result as our algorithm, which further verifies the optimality of our algorithm.
\end{itemize}

\section{ Preliminaries}
\subsection{Definition and proposition used in our model}

\begin{definition}[Strictly diagonally dominant]
Square matrix $A=(a_{ij})_{n\times n}$ is strictly diagonally dominant if the absolute value of its main diagonal element 
is greater than the sum of the absolute values of the other elements in the row, i.e., $|a_{ii}|>\sum_{j\not=i}|a_{ij}|$,
$i=1,\ldots,n$.
\end{definition}

\begin{proposition}\label{prop:diagonally-dominant}
Let $A=(a_{ij})_{n\times n}$ be a square matrix that is strictly diagonally dominant, then $\det(A)\neq 0$.
\end{proposition}

\begin{proof}
Suppose $\det(A)=0$, then the system of linear equations $Ax=0$ has non-zero solution $x=(x_1,x_2,\ldots,x_n)^{\mathrm{T}}$. 
Suppose the largest one of $|x_1|,\ldots,|x_n|$ is $|x_k|$, then $|x_k|>0$. Known by $\sum_{j=1}^na_{kj}x_j=0$, we have 
  \begin{align*}
     \Big|\sum_{j\neq k}a_{kj}x_j\Big| & =|-a_{kk}x_k|=|a_{kk}| |x_k|>|x_k| \sum_{j\neq k}|a_{kj}| \\
     & \geq \sum_{j\neq k}|a_{kj}||x_j|\geq \Big|\sum_{j\neq k}a_{kj}x_j\Big|,
  \end{align*}
which is a contradiction.
\end{proof}

\subsection{Friedkin-Johnsen Model}
We can represent a social trust network with a directed signed weighted graph $G=(V,E)$, where $V=\{1,2,\ldots,n\}$ is 
the set of nodes and $E$ is the set of arcs (directed edges). Each directed edge from $i$ to $j$, $(i,j)\in E$, is 
associated with the weight $\omega_{ij}$, which shows the extent of the relationship. Let $A=(a_{ij})$ be the 
adjacency matrix of $G$, where $a_{ij}=\omega_{ij}$ if $(i,j)\in E$, $a_{ij}=0$ otherwise. Let $D$ be a diagonal 
matrix with diagonal entry $d_{ii}=\sum_{j=1}^{n}|a_{ij}|$, and $L=D-A$ be the Laplacian matrix.

The Friedkin-Johnsen model \cite{Friedkin1990Social} is a famous dynamics model of continuous-valued opinion.
In this model, each node $i\in V$ has a fixed internal opinion $s_i$ that stays the same. During each 
iteration, the expressed opinion $z_i$ of node $i$ is updated as follows:
\begin{equation}
    z_i=\frac{s_i+\sum_{j\in N^+(i)}\omega_{ij}z_j}{1+\sum_{j\in N^+(i)}\omega_{ij}},
    \label{1}
\end{equation}
where the values of $s_i$ and $z_i$ are in the interval $[-1,1]$, $N^+(i)$ denotes the set of successors of node $i$, 
and $\omega_{ij}$ denotes the weight of the directed edge $(i,j)$. In a social game, the result of repeated average converges to 
the Nash equilibrium \cite{2015How}.

\subsection{Generalized Opinion Dynamics Model (GODM) For Social Trust Networks}
It can be seen from the Friedkin-Johnsen model that for each node $i$, its expressed opinions are the weighted average 
of its internal opinions and the opinions of the people around it. However, in real life, people's personalities are 
different, which lead to the different degrees they are affected by other people. So we develop a generalized opinion 
dynamics model for social trust networks.

The consensus cost function of the social game is computed as follows:
\begin{equation}
    c(z_i)\!=\!\alpha_i(z_i\!-\!s_i)^2\!+\!(1\!-\!\alpha_i)\!\!\sum_{j\in N^+(i)}\!\! |\omega_{ij}|(z_i\!-\!\sgn(i,j)z_j)^2,
    \label{2}
\end{equation}
where $\sgn(i,j)$ denotes the sign of nonzero $\omega_{ij}$, and $\alpha_i$ denotes the confidence index of node $i$. The 
confidence index $\alpha_i$ reflects how confident everyone in a social network is. In social networks, a higher $\alpha_i$ 
indicates that a person is more likely to listen to his or her own thoughts when expressing opinions, rather than being 
easily influenced by others. In a real social network, the $\alpha_i$ value is related to its importance in the network 
and how others evaluate it.

\begin{theorem}\label{thm:ab}
Let $\Lambda=\text{diag}\,(\alpha_1,\ldots,\alpha_n)$ be a diagonal matrix, where $\alpha_i$ is the confidence index of 
node $i$ in a social trust network, and let $L$ be the Laplacian matrix of the social trust network. Then $\Lambda+(I-\Lambda)L$ 
is invertible.
\end{theorem}
  
\begin{proof}
It is clear that 
\[
(\Lambda+(I-\Lambda)L)_{ii}=\alpha_i+(1-\alpha_i)\sum_{j\in N^+(i)}|\omega_{ij}|.
\] 
The sum of the absolute values of the off-diagonal elements of $\Lambda+(I-\Lambda)L$ is equal to the sum of the absolute 
values of the off-diagonal elements of $(I-\Lambda)L$, which is $(1-\alpha_i)\sum_{j\in N^+(i)}|\omega_{ij}|$. For 
$0<\alpha_i<1$, $\Lambda+(I-\Lambda)L$ is the diagonally dominant matrix, as desired.
\end{proof}

\begin{proposition}\label{pro:ab}
Let $s$ be the internal opinion vector, whose $i$-th component is $s_i$, and $z^*$ be the expressed opinion vector at the 
Nash equilibrium. We have $z^*=(\Lambda+(I-\Lambda)L)^{-1}\Lambda s$.  
\end{proposition}

\begin{proof}
At the Nash equilibrium, we compute expressed opinion $z_i^*$ that minimizes its cost for each node $i$, i.e., 
$c'(z_i^*)=0$. Hence,
\[
c'(z_i^*)\!\!=\!2\alpha_i(z_i^*-s_i)\!+2(1-\alpha_i)\!\!\!\sum_{j\in N^+(i)}\!\!|\omega_{ij}|(z_i^*-\sgn(i,j)z_j^*\!)\!\!=\!\!0. 
\]
It follows that
\begin{equation}
z_i^*=\frac{\alpha_i s_i+(1-\alpha_i)\sum_{j\in N^+(i)}\omega_{ij}z_j^*}{\alpha_i+(1-\alpha_i)\sum_{j\in N^+(i)}|\omega_{ij}| }.
\label{3}
\end{equation}
We rewrite Equation \eqref{3} as:
\[
\Big(\alpha_i+(1-\alpha_i)\sum_{j\in N^+(i)}|\omega_{ij}|\Big)z^*= \alpha_i s_i+(1-\alpha_i)\sum_{j\in N^+(i)}\omega_{ij}z_j^*.
\]
By introducing the opinion vectors $s$ and $z^*$, we can transform Equation \eqref{3} as it’s matrix form:
$(\Lambda+(I-\Lambda)L)z^*=\Lambda s$. According to the Theorem \ref{thm:ab}, $\Lambda+(I-\Lambda)L$ is invertible. Therefore, 
we obtain that $z^*=(\Lambda+(I-\Lambda)L)^{-1}\Lambda s$.
\end{proof}

\textbf{PageRank:}
PageRank is a technology which is calculated by search engines based on the hyperlinks between pages \cite{1998page}.
It is used to show the relevance and importance of web pages.

%
%
%
%

For a page $v_i$: $r(v_i)= \frac{1+d}{N}+d\big(\frac{r(u_1)}{c(u_1)}+\cdots+\frac{r(u_n)}{c(u_n)}\big)$, where $r(v_i)$ denotes the 
PageRank value of page $v_i$, $u_i$ refers to one of the pages that points to $v_i$, $c(u_i)$ denotes the number of edges 
that point to other pages for page $u_i$. $N$ is the number of nodes in the social network. $d$ is the damping coefficient, 
usually $d=0.85$. In a large network, the PageRank value for each point may be very small. So we normalized the PageRank 
values of points by dividing by the maximum PageRank value among all points.

\textbf{Mean evaluation:}
In social networks, each node plays two roles, the output node and the input node. As an input node, its edge weight and 
edge sign represent the evaluation that other nodes on it. Here, for a node $i$, we use mean evaluation to express the 
average evaluation that other nodes on it, so for node $i$, the calculation formula of mean evaluation that others on it 
can be expressed as: $m(i)=\frac{\sum_{j\in N^-(i)}\omega_{ij}}{d_i^-}$, where $m(i)$ denotes the mean evaluation of node 
$i$, $N^-(i)$ is the predecessors of node $i$, and $d_i^-$ denotes the amount of the predecessors of node $i$.

\textbf{Compute the confidence index $\alpha_i$:}
Here, let's consider two scenarios. First, the confidence index $\alpha_i$ is fixed, which doesn't change with the structure of the graph.
Second, the confidence index $\alpha_i$ is adjustable. Here we let the confidence index $\alpha_i$ weighted by PageRank value and mean evaluation, 
which is calculated by the following formula:
\begin{equation}
  \alpha_i=qm(i)+(1-q)r(i),  \label{4}
\end{equation}
where $m(i)$ and $r(i)$ are the mean evaluation and PageRank value of node $i$, respectively, and $0\leq q\leq 1$ is a 
constant. When $q=1$, the value of $\alpha_i$ is completely determined by the PageRank value, and when $q=0$, the value
of $\alpha_i$ is determined by what other people say about it. In this case, let's take $q=0.5$.

In real social trust networks, such as a network of users’ reviews of a website. Due to the trust and distrust between 
users, this may result in a negative mean evaluation of a user. And the PageRank value may be very small, which may 
result in the calculated $\alpha_i$ value, the confidence index, being negative. This is unreasonable, so we introduce 
the Rectified Linear Unit, which restricts the value of $\alpha_i$ to $[0,1]$. So Equation \eqref{4} becomes the following:
\begin{equation}
\alpha_i=RELU(qm(i)+(1-q)r(i)), \label{5}
\end{equation}
where $RELU(x)=\max\{0,x\}$. From Equation \eqref{5}, we can see that the more important a person is in the network and 
the higher people's evaluation of him, the higher his confidence index will be.

\begin{definition}[Overall opinion]
The overall opinion $p(z^*)$ of a social trust network is the sum of the expressed opinions at Nash equilibrium:
$p(z^*)=\sum_i z^*_i=\bm{1}^{\mathrm{T}} (\Lambda+(I-\Lambda)L)^{-1}\Lambda s$, where $\bm{1}$ is the all-ones column vector.
\end{definition}

\section{Internal Opinion Maximization in Social Trust Networks}

\begin{definition}[Internal opinion problem]
Given a social trust network $G$, the internal opinion vector $s$, and the budget amount $\mu$, our goal is to find out 
the modification of internal opinions $\Delta s$ to maximize the overall opinion $p(z^*)$:
\begin{align*}
  & \max~\bm{1}^{\mathrm{T}} (\Lambda+(I-\Lambda)L)^{-1}\Lambda(s+\Delta s), \\
  & \text{s.t.}~
    \begin{cases}
      -1\leq s_i+\Delta s_i\leq1, ~(i=1,2,\ldots,n), \\
      \|\Delta s\|_1\leq\mu.
    \end{cases}            
\end{align*}
\end{definition}

According to the Proposition \ref{pro:ab}, the expressed opinion vector is in the column space of $\Lambda+(I-\Lambda)L$ 
and the coordinates are stored in the vector of internal opinions. Then, we name $g=\bm{1}^{\mathrm{T}} (\Lambda+(I-\Lambda)L)^{-1}\Lambda$ 
the contribution index vector. As $p(z^*)=\bm{1}^{\mathrm{T}} (\Lambda+(I-\Lambda)L)^{-1}\Lambda s=gs$, the $i$-th element of $g$  
represents the contribution of the internal opinion of node $i$ to the overall opinion. Specifically, the 
contribution index is defined as follows.

\begin{definition}[Contribution index]
The contribution index $g_i$ of node $i$ represents how $i$'s internal opinion contributes to the overall opinion and is 
quantified by:  
\[
  g_i=(\bm{1}^{\mathrm{T}} (\Lambda+(I-\Lambda)L)^{-1}\Lambda)_i. 
\]
\end{definition}

To maximize the overall opinion, we change the internal opinion with the largest absolute value of the contribution index of the node, then the maximum benefit 
can be obtained under the same budget. Our approach is summarized in Algorithm 1. It is further proved that Algorithm 1 
outputs the optimal result of the internal opinion problem.

\begin{algorithm}  
\SetKwInput{Input}{Input}
\SetKwInOut{Output}{Output}
\SetAlgoLined
\Input{Social trust network $G=(V,E)$; Budget $\mu$; Randomly generate internal opinions $s$.}
\Output{The modification $\Delta s$.}  
\Begin{
  Compute $g_i=(\bm{1}^{\mathrm{T}} (\Lambda+(I-\Lambda)L)^{-1}\Lambda)_i$\;
  Set $sign=\frac{g_i}{|g_i|}$\;
    \While{$\mu>0$}{
      Find the largest $|g_i|$, $cost=1-sign \cdot s_i$\;
        \eIf{$cost>\mu$}{
          $\Delta s=sign\cdot \mu,~\mu=0$\;
        }{
          $\Delta s=sign\cdot cost,~\mu=\mu-cost$\;
      }
    }
}
\caption{Solving Internal Opinion Problem}
\end{algorithm}

\begin{theorem}
Algorithm 1 outputs the optimal result of the internal opinion problem.
\end{theorem}

\begin{proof}
According to the Algorithm 1, we know the benefit is $g\Delta s$. For any intervened node $i$, we can 
adjust the corresponding modification $\Delta s_i$ to $(\Delta s_i-v)$ if $\Delta s_i>0$, or $(\Delta s_i+v)$ otherwise, 
where $v>0$. Then we have an unused budget $v$ and randomly find a node $j$ that is not intervened before. As the internal opinion of node $j$ is changed, 
the benefit becomes $g \Delta s+(|g_j|-|g_i|)v$. We know from Algorithm 
1 that $|g_j|\leq|g_i|$, thus $g\Delta s\geq g \Delta s+(|g_j|-|g_i|)v$.
\end{proof}

\section{Solve internal opinion maximization problem with the ADMM algorithm}
\subsection{ADMM algorithm}
Alternating Direction Method of Multipliers (ADMM) is a convergent algorithm combining the factorability of dual ascending 
and the advantages of multiplier method, which was proposed by Boyd \cite{2010Distributed}. When there are multiple variables
in the target of a convex optimization problem, it is difficult to solve, so we use variable separation to divide the problem
into several simple subproblems. Then we can solve these subproblems in parallel and coordinate the subproblems to get the 
global solution of original problem. 

\subsection{Steps of the ADMM algorithm}
Suppose the optimization problems are as follows:
\begin{align*}
& \min_{x,z}~ f(x)+g(z) \\
& \text{s.t.}~Ax+Bz=c,
\end{align*}
where $x\in \mathbb{R}^n$, $z\in \mathbb{R}^m$, $A\in \mathbb{R}^{p\times n}$, $B\in \mathbb{R}^{p\times m}$,
$c\in \mathbb{R}^p$. Its augmented Lagrangian function is as follows:
\[
L_\rho(x,z,y)=f(x)+g(z)+y^{\mathrm{T}}(Ax+Bz-c)+\frac{\rho}{2}\|Ax+By-c\|_2^2. 
\]
The update iteration form of ADMM is as follows:
\begin{align*}
x^{k+1} & :=\arg\min_xL_\rho(x,z^k,y^k), \\
z^{k+1} & :=\arg\min_zL_\rho(x^{k+1},z,y^k), \\
y^{k+1} & :=y^k+\rho(Ax+Bz-c),
\end{align*}
where $\rho>0$, $c\in \mathbb{R}^p$ is a dual variable, and this is the original iteration form of ADMM algorithm.

\subsection{Solve the internal opinion problem with the ADMM algorithm}
The internal opinion problem we want to solve is as follows:
\begin{align*}
  & \max~\bm{1}^{\mathrm{T}} (\Lambda+(I-\Lambda)L)^{-1}\Lambda(s+\Delta s) \\
  & \text{s.t.}~
    \begin{cases}
      -1\leq s_i+\Delta s_i\leq 1,~(i=1,2,\ldots,n), \\
      \|\Delta s\|_1\leq\mu.
    \end{cases}             
\end{align*}
Because $\bm{1}^{\mathrm{T}} (\Lambda+(I-\Lambda)L)^{-1}\Lambda s$ is a constant, we can rewrite the problem as:
\begin{align*}
  & \max~\bm{1}^{\mathrm{T}} (\Lambda+(I-\Lambda)L)^{-1}\Lambda \Delta s, \\
  & \text{s.t.}~
    \begin{cases}
      -1-s_i\leq \Delta s_i\leq1-s_i, ~(i=1,2,\ldots,n), \\
      \|\Delta s\|_1\leq\mu.
    \end{cases}               
\end{align*} 
We can translate the problem into the following form:
\begin{align*}
& \min~ l(x)+\lambda \|x\|_1, \\
& \text{s.t.}~ a\leq x\leq b,
\end{align*}
where $l(x)=-\bm{1}^{\mathrm{T}} (\Lambda+(I-\Lambda)L)^{-1}\Lambda$, $x=\Delta s$, $a_i=-1-s_i$, $b_i=-1+s_i$.
We then further translate this into the following form that can be solved by ADMM:
\begin{align*}
  & \min~ l(x)+\lambda \|z\|_1, \\  
  & \text{s.t.}~
     \begin{cases}
       a\leq x\leq b, \\
       x=z.
     \end{cases}       
\end{align*}
In this case, the variables $x$, $z$ and $u$ are updated as follows:
\begin{align*}
x^{k+1} & :=\arg\min_x\big(l(x)+\frac{\rho}{2}\|x-z^k+u^k\|_2^2\big), \\   
z^{k+1} & :=S_{\lambda/\rho}(x^{k+1}+u^k), \\   
u^{k+1} & :=u^k+x^{k+1}-z^{k+1}.   
\end{align*}
From the above formula, the analytic solution of each step iteration can be written:
\begin{align*}
x_i^{k+1} & :=
  \begin{cases}
    a_i, & \text{if}~ z_i^k-u_i^k-\frac{A_i}{\rho}<a_i, \\
    z_i^k-u_i^k-\frac{A_i}{\rho}, & \text{if}~ z_i^k-u_i^k-\frac{A_i}{\rho}<b_i, \\
    b_i, & \text{if}~ b_i<z_i^k-u_i^k-\frac{A_i}{\rho},
  \end{cases} \\  
  z_i^{k+1} & :=
    \begin{cases} 
      z_i^{k+1}-u_i^k-\frac{\lambda}{\rho}, & \text{if}~ z_i^{k+1}-u_i^k>\frac{\lambda}{\rho}, \\
      0, & \text{if}~ -\frac{\lambda}{\rho}<z_i^{k+1}-u_i^k<\frac{\lambda}{\rho}, \\
      z_i^{k+1}-u_i^k+\frac{\lambda}{\rho}, & \text{if}~ z_i^{k+1}-u_i^k<-\frac{\lambda}{\rho},
      \end{cases} \\  
  u_i^{k+1} & :=u_i^k+x_i^{k+1}-z_i^{k+1}. 
\end{align*}
The algorithm converges to the global optimal solution, so after a certain number of iterations, we can get the optimal result.

\section{Experiments}
In this section, we conduct a series of experiments to evaluate the proposed method. We carry out experiments from two 
aspects. First, we compare the results obtained by our model with those in \cite{2020Opinion}. To further illustrate 
the superiority of our model, our method is compared with other heuristic methods, and the experimental results show 
that our method has overwhelming advantages.

\subsection{Datasets}
The datasets we used in the experiments are as follows: (i) Alpha; (ii) OTC \cite{Srijan2017Edge}, 
\cite{2018REV2}; and we normalize the trust values (i.e., edge weights) to the interval $[-1,1]$ on Alpha and OTC. 
(iii) Elec; (iv) Rfa \cite{2014Exploiting}; and the relationships in the last two networks are closely related to trust. 

\subsection{Internal opinion initialization}
We randomly select values that match a particular distribution to initialize the internal opinion vector $s$ to simulate 
different situations. For each network, we use three ways to initialize the internal opinions. (i) We initiate the internal opinions 
to follow a uniform distribution (i.e., $s \sim U( -1, 1)$). (ii) We initiate the internal opinions to follow a standard normal distribution 
(i.e., $s \sim N(0, 1)$). (iii) We initiate the internal opinions of a node to positively correlated with the column connectivity of 
that node (i.e., $s_i \propto\sum_j |a_{ji}|$).

\subsection{The value of $\alpha_i$}
We experimented with two different ways of selecting value of $\alpha_i$: (i) Fixed $\alpha_i$. In this case, we fixed 
the value of $\alpha_i$, which means that everyone's confidence index is the same in the network. (ii) Adjusted $\alpha_i$. 
In this case, $\alpha_i$ changes with the internal structure of each group. When $\alpha_i\in\{2/3, 1/2, 1/3, 1/4\}$, we 
carried out a set of tests, and the experimental results are shown in Table \ref{tab:cap}. 

As can be seen from Table \ref{tab:cap}, when $\alpha_i=1/2$, that is, the proportion of the opinions of everyone inside 
and those of the people around is the same, and the experimental results are consistent with those in \cite{2020Opinion}. 
When $\alpha_i>1/2$, the experimental result is lower than the result in \cite{2020Opinion}, and when $\alpha_i<1/2$, 
the experimental result is higher than the result in \cite{2020Opinion}. This situation is also reasonable in real life. 
A larger $\alpha_i$ indicates that the opinions expressed by each person are less likely 
to be influenced by those around them. The smaller $\alpha_i$ is, the more susceptible the expressed influence is to those 
around it. Therefore, changing the internal opinions of a small number of people will affect the overall opinions of the 
whole network to change greatly. 
       
\begin{table}[t]
  \centering
    \caption{Average benefit compared with different $\alpha_i$} 
      \begin{tabular}{|c|c|c|c|c|c|}
        \hline  
        \multirow{4}{*}{Alpha} & $\alpha_i$ & $2/3$ & $1/2$ & $1/3$ & $1/4$ \\
        \cline{2-6}  
        ~& OMSTN \cite{2020Opinion} & 282 & 277 & 265 & 281 \\
        \cline{2-6} 
        ~& Fixed $\alpha_i$ & 254 & 277 & 293 & 343 \\
        \cline{2-6}
        ~& Adjusted $\alpha_i$ & 481 & 461 & 429 & 479 \\
        \hline  
        \multirow{4}{*}{OTC} & $\alpha_i$ & $2/3$ & $1/2$ & $1/3$ & $1/4$ \\
        \cline{2-6}  
        ~& OMSTN & 306 & 318 & 316 & 304 \\
        \cline{2-6} 
        ~& Fixed $\alpha_i$ & 270 & 318 & 367 & 386 \\
        \cline{2-6}
        ~& Adjusted $\alpha_i$ & 546 & 600 & 578 & 555 \\
        \hline 
        \multirow{4}{*}{WikiElec} & $\alpha_i$ & $2/3$ & $1/2$ & $1/3$ & $1/4$ \\
        \cline{2-6}  
        ~& OMSTN & 1165 & 1115 & 1144 & 1149 \\
        \cline{2-6} 
        ~& Fixed $\alpha_i$ & 974 & 1115 & 1305 & 1394 \\
        \cline{2-6}
        ~& Adjusted $\alpha_i$ & 1538 & 1468 & 1499 & 1510 \\
        \hline 
        \multirow{4}{*}{WikiRfa} & $\alpha_i$ & $2/3$ & $1/2$ & $1/3$ & $1/4$ \\
        \cline{2-6}  
        ~& OMSTN & 1010 & 1000 & 1023 & 968 \\
        \cline{2-6} 
        ~& Fixed $\alpha_i$ & 906 & 1000 & 1121 & 1098 \\
        \cline{2-6}
        ~& Adjusted $\alpha_i$ & 1340 & 1341 & 1375 & 1293 \\
        \hline 
      \end{tabular}     
  \label{tab:cap}
\end{table}

\subsection{Comparative Methods}
In order to better reflect the advantages of our results, as with \cite{2020Opinion}, we compared it to several other three 
heuristics. (i) Rand. We sort the nodes randomly. (ii) Trust. 
We represent the trust sum of the nodes by calculating the sum of the corresponding columns of the adjacency matrix. Nodes with a large 
amount of trust may have a strong ability to influence the opinions of other nodes. Therefore, we sort the nodes in descending 
order of the node trust sum. (iii) IO. If we can persuade those negative internal opinions to have positive internal opinions, 
then the overall opinion may increase. Therefore, we take the internal opinions of the node and sort the corresponding nodes in ascending order.
In order to ensure the fairness of the experimental results, we reproduced the other methods mentioned in \cite{2020Opinion} 
and carried out the experiment on the same computer. In the experiment, we set the parameter $\mu=200$, and the result is 
shown in Table \ref{tab:cap2} and Fig. \ref{fig:res1}.

As can be seen from the results, our approach has an overwhelming advantage over all datasets. A part of the performance curves 
are shown in Fig. \ref{fig:res}. The $x$-axis shows the budget amount, the 
left $y$-axis shows the overall opinion, and the right $y$-axis shows the unit benefit.

  \begin{table}[t]
    \centering
    \caption{Average benefit compared with other heuristics} \label{tab:cap2}
    \begin{tabular}{|c|c|c|c|c|c|}
    \hline  
    \multirow{5}{*}{Uniform} &~ & Alpha & OTC & WikiElec & WikiRfa \\
    \cline{2-6}  
    ~& Ours & 463 & 532 & 1550 & 1340 \\
    \cline{2-6} 
    ~& Rand & 41 & 21 & 13 & 29 \\
    \cline{2-6}
    ~& Trust & 225 & 83 & 275 & 298 \\
    \cline{2-6}
    ~& IO & 23 & 25 & 144 & 34 \\
    \hline  
    \multirow{5}{*}{Normal}&~& Alpha & OTC & WikiElec & WikiRfa \\
    \cline{2-6}  
    ~& Ours & 452 & 578 & 1581 & 1341 \\
    \cline{2-6} 
    ~& Rand & 26 & 20 & 87 & 14 \\
    \cline{2-6}
    ~& Trust & 250 & 112 & 205 & 378 \\
    \cline{2-6}
    ~& IO & 36 & 23 & 233 & 15 \\
    \hline 
    \multirow{5}{*}{Degree}&~& Alpha & OTC & WikiElec & WikiRfa \\
        \cline{2-6}  
        ~& Ours & 454 & 569 & 1545 & 1375 \\
        \cline{2-6} 
        ~& Rand & 18 & 14 & 29 & 40 \\
        \cline{2-6}
        ~& Trust & 250 & 92 & 297 & 329 \\
        \cline{2-6}
        ~& IO & 260 & 120 & 8 & 174 \\
        \hline 
    \end{tabular}
  \end{table}
      
\begin{figure}[t]
    \centerline{\epsfig{figure=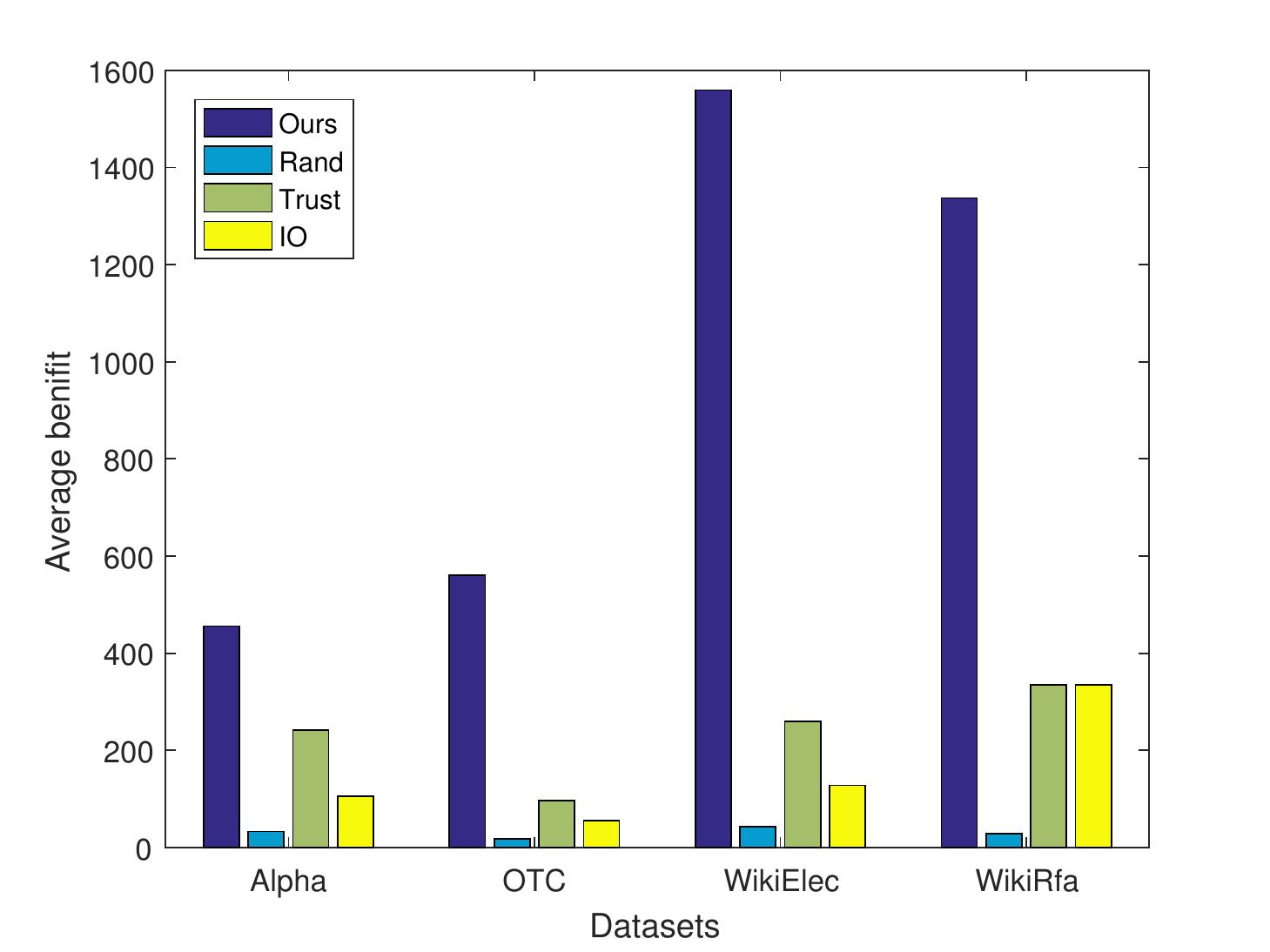,width=7cm}}
    \caption{Average benefit of different methods on four datasets.}
    \label{fig:res1}
\end{figure}


\begin{figure}[t]
  \begin{minipage}[b]{0.3\linewidth}
    \centering
    \centerline{\epsfig{figure=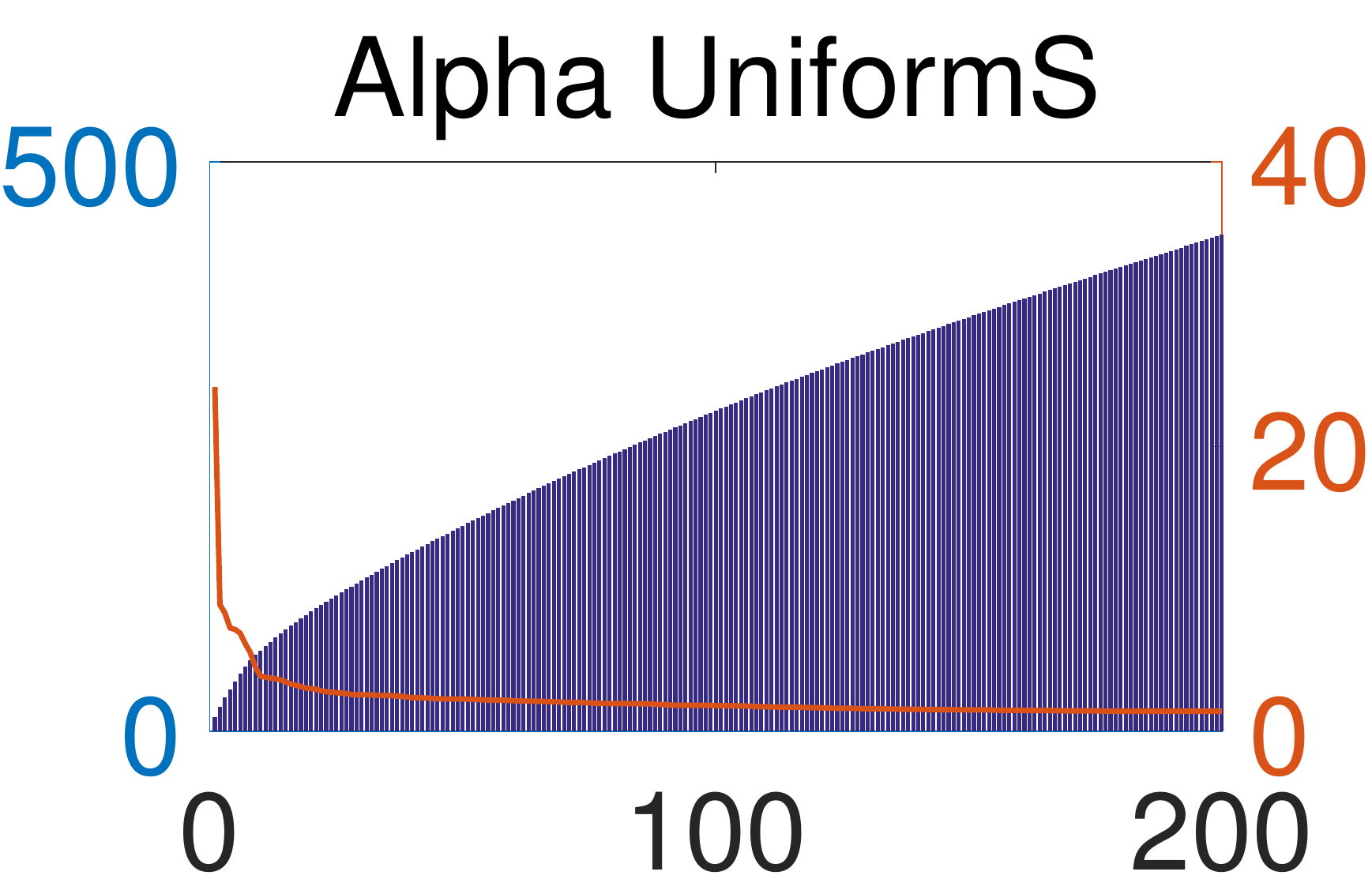,width=2.5cm}}
  \end{minipage}
  \begin{minipage}[b]{0.3\linewidth}
    \centering
    \centerline{\epsfig{figure=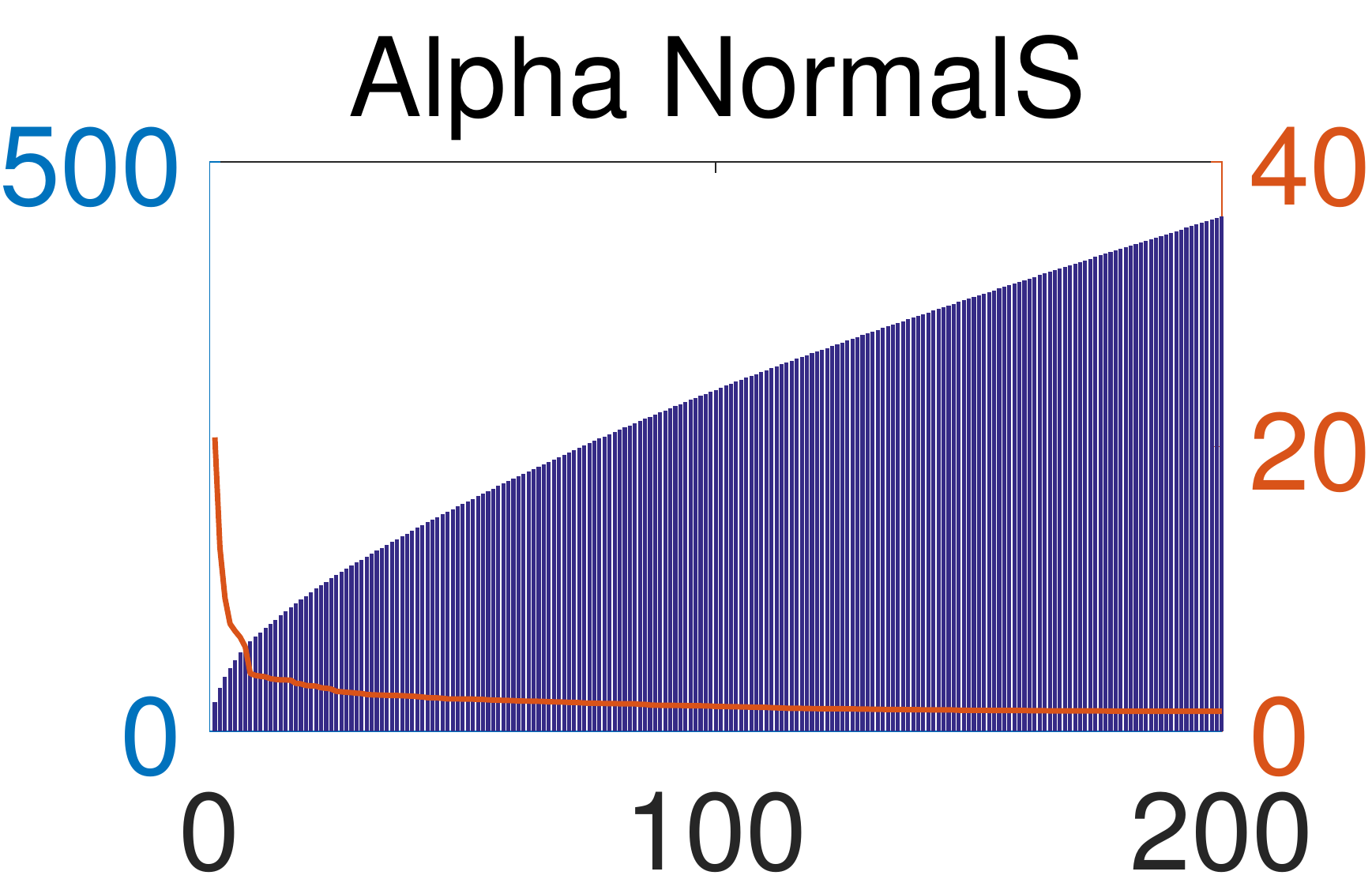,width=2.5cm}}
  \end{minipage}
  \begin{minipage}[b]{.3\linewidth}
    \centering
    \centerline{\epsfig{figure=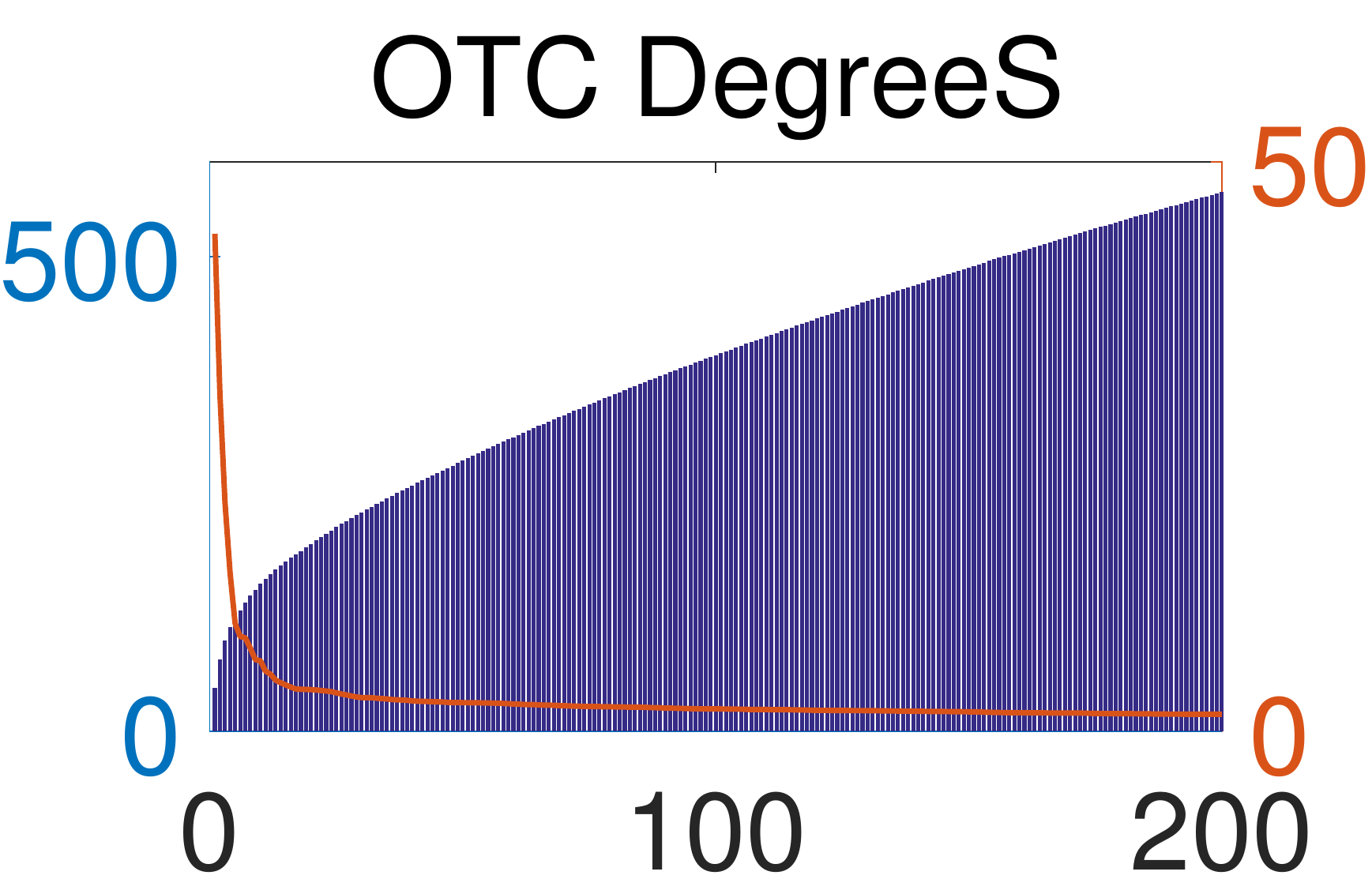,width=2.5cm}}
  \end{minipage}
  \caption{ A part of the performance curves.}
  \label{fig:res}
\end{figure}

\section{Conclusion}
In this paper, we have proposed a generalized opinions dynamics model for social trust networks to solve the internal 
opinion maximization problem. Different from other methods, our model GODM introduces a new confidence index and proposes 
a novel algorithm, which proved to be optimal. Extensive experimental results on four datasets demonstrate the effectiveness 
of our model.

\bibliographystyle{IEEEbib}
\bibliography{icme2021template}

\begin{thebibliography}{10}

\bibitem{2015Online}
Shuo Chen, Ju~Fan, Guoliang Li, Jianhua Feng, Kian~Lee Tan, and Jinhui Tang,
\newblock ``Online topic-aware influence maximization,''
\newblock {\em Proceedings of the VLDB Endowment}, 2015.

\bibitem{2020Opinion}
Pinghua Xu, Wenbin Hu, Jia Wu, and Weiwei Liu,
\newblock ``Opinion maximization in social trust networks,''
\newblock {\em Proceedings of the International Joint Conference on Artificial
  Intelligence}, 2020.

\bibitem{Liu2018}
Xiangnan~Kong Xinyue~Liu and Philip~S Yu,
\newblock ``Active opinion maximization in social networks,''
\newblock {\em In Proceedings of the 24th ACM SIGKDD International Conference
  on Knowledge Discovery \& Data Mining}, pp. 1840--1849, 2018.

\bibitem{2018Opinion}
Rediet Abebe, Jon Kleinberg, David Parkes, and Charalampos~E Tsourakakis,
\newblock ``Opinion dynamics with varying susceptibility to persuasion,''
\newblock 2018.

\bibitem{Friedkin1990Social}
Noah~E. Friedkin and Eugene~C. Johnsen,
\newblock ``Social influence and opinions,''
\newblock {\em Journal of Mathematical Sociology}, vol. 15, no. 3--4, pp.
  193--206, 1990.

\bibitem{2015How}
David Bindel, Jon Kleinberg, and Sigal Oren,
\newblock ``How bad is forming your own opinion?,''
\newblock {\em Games \& Economic Behavior}, vol. 92, pp. 248--265, 2015.

\bibitem{1998page}
Page Lawrence, Brin Sergey, Motwani Rajeev, and Winograd Terry,
\newblock ``The pagerank citation ranking:bringing order to the web,''
\newblock 1998.

\bibitem{2010Distributed}
Stephen Boyd, Neal Parikh, Eric Chu, Borja Peleato, and Jonathan Eckstein,
\newblock ``Distributed optimization and statistical learning via the
  alternating direction method of multipliers,''
\newblock {\em Foundations \& Trends in Machine Learning}, vol. 3, no. 1, pp.
  1--122, 2010.

\bibitem{Srijan2017Edge}
Srijan Kumar, Francesca Spezzano, V.S. Subrahmanian, and Christos Faloutsos,
\newblock ``Edge weight prediction in weighted signed networks,''
\newblock in {\em IEEE International Conference on Data Mining}, 2017.

\bibitem{2018REV2}
Srijan Kumar, Bryan Hooi, Disha Makhija, Mohit Kumar, and V.~S. Subrahmanian,
\newblock ``Rev2: Fraudulent user prediction in rating platforms,''
\newblock in {\em the Eleventh ACM International Conference}, 2018.

\bibitem{2014Exploiting}
Robert West, Hristo~S Paskov, Jure Leskovec, and Christopher Potts,
\newblock ``Exploiting social network structure for person-to-person sentiment
  analysis,''
\newblock {\em Eprint Arxiv}, 2014.

\end{thebibliography}

\end{document}